\documentclass[10pt]{article}
\usepackage{amssymb, amsmath, amsthm, color,graphicx, amsfonts,a4}
\usepackage{amscd}

\definecolor{Red}{cmyk}{0,1,1,0}

\definecolor{verde}{cmyk}{1,0,1,0}

\definecolor{azul}{cmyk}{1,1,0,0}

\newcommand{\cima}[2]{\genfrac{}{}{0pt}{3}{#1}{#2}}
\numberwithin{equation}{section}

\def\cal{\mathcal}

\newcommand{\ra}{\rightarrow}

\def\Ed{{\mathbb{E}}}

\def\Zd{\mathbb{Z}^d}

\newcommand{\E}{\mathbb{E}}
\newcommand{\N}{\mathbb{N}}

\newcommand{\Z}{\mathbb{Z}}
\renewcommand{\P}{\mathbb{P}}

\newcommand{\V}{\mathbb{V}}
\renewcommand{\a}{\alpha}

\newcommand{\g}{\gamma}

\newcommand{\e}{\varepsilon}

\newcommand{\s}{\sigma}
\renewcommand{\o}{\omega}

\renewcommand{\l}{\lambda}

\renewcommand{\L}{\Lambda}

\newcommand{\m}{\mu}

\newcommand{\be}{\begin{equation}}
\newcommand{\ee}{\end{equation}}

\newtheorem{teorema}{Theorem}

\newtheorem{lema}{Lemma}
\newtheorem{corolario}{Corollary}

\begin{document}
\title{Percolation of words on $\Zd$ with long range connections}
\author{Bernardo N.B. de Lima\footnote{ Departamento de Matem{\'a}tica, Universidade Federal de Minas Gerais, Av. Ant\^onio
Carlos 6627 C.P. 702 CEP30123-970 Belo Horizonte-MG, Brazil} ,
R\'emy Sanchis$^{*}$, Roger W.C. Silva\footnote{ Departamento de
Estat\'\i stica, Universidade Federal de Minas Gerais, Av.
Ant\^onio Carlos 6627 C.P. 702 CEP30123-970 Belo Horizonte-MG,
Brazil}} \maketitle
\begin{abstract}
Consider an independent site percolation model on $\Z^d$, with
parameter $p \in (0,1)$, where all long range connections in the
axes directions are allowed. In this work we show that given any
parameter $p$, there exists and integer $K(p)$ such that all
binary sequences (words) $\xi \in \{0,1\}^{\N}$ can be seen
simultaneously, almost surely, even if all connections whose
length is bigger than $K(p)$ are suppressed. We also show some results concerning the question how
$K(p)$ should scale with $p$ when $p$ goes to zero. Related results are also
obtained for the question of whether or not almost all
words are seen.
\end{abstract}
{\footnotesize Keywords: percolation of words; truncation's
question }

\section{Introduction and Notation}
The problem of percolation of words was introduced in \cite{BK} and is formulated as follows. Let $G=(\V,\E)$ be a graph
with a countably infinite vertex set $\V$. Consider site
percolation on $G$; to each site $v \in \V$ we associate
a Bernoulli random variable $X(v)$, which takes the values 1 and 0
with probability $p$ and $1-p$ respectively. This can be done
considering the probability space $(\Omega,\mathcal{F},\P_p)$,
where $\Omega=\{0,1\}^{\V}$, $\mathcal{F}$ is the $\sigma$-algebra
generated by the cylinder sets in $\Omega$ and
$\P_p=\prod_{v\in\V}\mu(v)$ is the product of Bernoulli measures
with parameter $p$, in which the configurations $\{X(v),v \in
\mathbb{V}\}$ takes place. We denote a typical element of $\Omega$
by $\omega$ and sometimes we write $X(\omega,v)$ instead of $X(v)$
to indicate that $X(v)$ depends on the configuration. When
$X(v)=1\,\,(X(v)=0)$ we say that $v$ is ``occupied'' ($v$ is
``vacant'', respectively). A path $\gamma$ on $G$ is a sequence
$v_1,v_2,\dots$ of vertices in $\V$, such that $v_i\neq v_j,\
\forall\ i\neq j$ and $v_{i+1}$ is a nearest neighbor of $v_i$,
for all $i$; that is the edge $\langle v_i,v_{i+1}\rangle$ belongs
to $\E$.

Let $\Xi=\{0,1\}^\N$. A semi-infinite binary sequence $\xi
=(\xi_1,\xi_2,\dots)\in\Xi$ will be called a {\it word}. Given a
word $\xi\in\{0,1\}^\N$, a vertex $v\in\V$ and a configuration
$\omega\in\Omega$, we say that the word $\xi$ {\it is seen in the
configuration $\o$ from the vertex $v$} if there is a self
avoiding path $\langle v=v_0,v_1,v_2\dots\rangle$ such that
$X(v_i)=\xi_i$, $\forall i=1,2,\dots$. Note that the state of $v$
is irrelevant. For fixed $\o\in\Omega$ and $v\in\V$, we will consider
the random sets $S_v(\o)=\{\xi\in\Xi ;\xi \mbox{ is seen in } \o
\mbox{ from } v \}$ and $S_{\infty}(\o)=\cup_{v\in\V} S_v(\o)$.
An interesting problem is to describe in which circumstances the
events $\{\o\in\Omega; S_\infty(\o)=\Xi\}$ and $\{\o\in\Omega ;
\exists v\in\V\mbox{ with } S_v(\o)=\Xi\}$ occur almost surely. Whenever either one of these occur, we say that {\it all words are seen}.

From a different perspective, if one suppose that the sequence of
digits in the word $\xi$ is a sequence of independent Bernoulli
random variables with parameter $\alpha$, i.e. each word $\xi$
take its values in the probability space $(\Xi,{\cal
A},\mu_\alpha)$, where ${\cal A}$ is the $\s$-algebra generated by
the cylinder sets in $\Xi$, and $\mu_{\alpha}=\prod_{n\in\N} \m(n)
$ is the product of Bernoulli measures with parameter $\alpha$,
another questions arise, namely 
whether the event
$$\{\omega\in\Omega ; \mu_\alpha
(S_\infty(\omega))=1\}$$
occurs almost surely. Whenever this occurs, we say that {\it almost all words are seen } or that {\it the random word percolates}.

In general, the problem of seeing all words is significantly harder
than the one of seeing almost all words. For instance, it is known that for $d\geq 3$ and $p=1/2$, almost all words are seen on $\Z^d$ with nearest neighbours whereas, in \cite{BK}, it is shown that it is possible to see all words on $\Z^d$,
a.s. for $d\geq 10$, but for $d<10$ this question
remains open (see Theorem 1 and Open Question 2 in
\cite{BK}). One should remark that, in general, seeing almost all words does not imply that all word are seen. For instance, Theorem 5 of  \cite{BK} gives
an example of a tree where we can see $\mu_\frac{1}{2}$-almost all
words but not all words are seen, $\P_{\frac{1}{2}}-$a.s.

In \cite{KSZ1}, it is shown that $\mu_\alpha$-almost all words are seen (with
$\alpha\in (0,1)$) on the triangular lattice,
$\P_{\frac{1}{2}}-a.s.$ (remember that in the triangular lattice
$p_c=\frac{1}{2}$, so it is not possible to see all words). In \cite{KSZ2}, it is proved that on the closed packed graph of
$\Z^2$ for $p\in (1-p_c(\Z^2), p_c(\Z^2))$ all words are seen
$\P_p-a.s.$

In the present paper we are concerned with the graph $G_K=(\V ,\E_K )$, in which $\V= \Z^d,
d\geq 2$ and where all long-range edges parallel to the coordinate axes are allowed, that is
\[\E_K=\{\langle (x,y)\rangle\subset \Z^d\times\Z^d :
\exists i\in\{1,\dots,d\}\mbox{ such that }\]
\[ 0<|x_i -y_i|\leq K \mbox{ and
} x_j=y_j,\forall j\neq i\}.\]

The graph $G_K$ can be seen as a truncation of a non-locally finite graph $G=(\V,\E)$, where
\[\E=\{\langle (x_1,\dots,x_d)(y_1,\dots,y_d)\rangle\subset \Z^d\times\Z^d :
\exists i\in\{1,\dots,d\}\mbox{ such that }\]
\[x_i\neq y_i \mbox{ and
} x_j=y_j,\forall j\neq i\},\] that is, $G_K$ can be obtained from $G$ by 
erasing all bonds whose length is larger than $K$.

In a previous paper by one of us (see \cite{L}), it is shown that for all $p\in (0,1)$ there exists a positive
integer $K=K(p)$ such that, on the graph $G_K$, 
\begin{equation}\label{boletim}
\P_p\left(\cup_{v\in\V}\{\o\in\Omega ;\xi\mbox{ is seen in
}\o\mbox{ from }v\mbox{ on }G_K\}\right)=1, \ \forall\xi\in \Xi.
\end{equation}

Moreover, (\ref{boletim}) implies that for the same $K(p)$ it is
possible to see $\mu_\alpha$-almost all words (with $\alpha\in
(0,1)$) on $G_K$, but (\ref{boletim}) does not imply that it is
possible to see all words.

Concerning related questions, we would like to single out the paper \cite{GLR}, where similar questions are considered,  but in $\Z$ instead of $\Z^d,\ d\geq2$. One of the results proved therein is that when $K=2$, not all words are seen $\P_{\frac{1}{2}}$-a.s.

In Section \ref{todas} we prove that there is a constant $K(p)$ such that, with positive probability, all words are seen on $G_K$ from a given vertex and state some results on the scaling of the constant $K(p)$ as $p\searrow 0$. In Section \ref{quasetodas}
we state a result on the scaling of the constant $K(p)$ for which $\mu_\alpha$-almost all words are seen on $G_K$. In Section \ref{sfinal}, we make some final remarks concerning the scaling behaviour for ordinary percolation and state some conjectures and open questions.

\section{All words can be seen}\label{todas}

This first result generalizes that of \cite{L}, showing that {\it all} words are seen on $G_K$ for sufficiently large $K$.

\begin{teorema}\label{box}
For all $p\in (0,1)$, there exists a positive integer $K=K(p)$,
such that
$$\P_p\{\omega\in\Omega; S_0(\omega)=\Xi\mbox{ on }G_K\}>0.$$
Equivalently
\be\label{seen}\P_p\left(\bigcup_{v\in\V}\{\omega\in\Omega; S_v(\omega)=\Xi\mbox{ on }G_K\}\right)=1.\ee
\end{teorema}
\begin{proof}
For any given $n\in\N$ and $x=(x_1,\dots,x_d)\in\Z^d$, let
\be\label{particao}
\L_x(n)=\{y=(y_1,\dots,y_d)\in\Z^d; 0\leq y_i-nx_i\leq
n-1,\forall\ i=1,\dots ,d\}
\ee
be a hypercubic box of side $n$. We
observe that for any $n\in\N$, the set of boxes $\{\L_x(n);
x\in\Z^d\}$ is a partition of $\Z^d$.

Consider a renormalized lattice, isomorphic to $\Z^d$, whose sites
are the boxes $\{\L_x(n); x\in\Z^d\}$. Given a configuration
$\o\in\Omega$, we declare each box as  ``good'', in the
configuration $\o$, if all lines have at least one occupied site
and one vacant site. To be precise, the box $\L_x(n)$ will be
``good'' if, for all $i\in\{1,\dots,d\}$ and for all finite
sequence $(l_j)_j$ with $l_j\in\{1,\dots ,n\}$ and
$j\in\{1,\dots,d\}-\{i\}$ there exists $z,w\in L(i,(l_j)_j)$ such
that $X(\o,z)=1$ and $X(\o,w)=0$, where
$$L(i,(l_j)_j)=\{y\in\L_x(n);y_j=l_j,\forall
j\in\{1,\dots,d\}-\{i\}\}$$ are the lines of $\L_x(n)$.

Consider the events $$A_x(n)=\{\o\in\Omega;\mbox{  the box
}\L_x(n)\mbox{ is good in }\o\}.$$ It is clear that all events of
the collection $\{A_x(n);x\in\Z^d\}$ are independent and have the
same probability. A rough estimate for a lower bound of this
probability gives

\begin{equation}
\label{goodbox} \P_p(A(n))=1-\P_p(A(n)^c )\geq
1-dn^{d-1}(p^n+(1-p)^n).
\end{equation}

Then,$$\lim_{n\rightarrow\infty}\P_p(A(n))=1,\ \forall
p\in(0,1).$$

Now, for $p\in(0,1)$ fixed, let
$N=N(p)=\min\{n\in\N;\P_p(A_n)>p_c(\Z^d) \}$, where $p_c(\Z^d)$ is
the ordinary independent nearest neighbour site percolation threshold for $\Z^d$. Then the origin of the
renormalized lattice will percolate with strictly positive
probability, that is, there is an infinite path
$(\Lambda_{x_0}(N),\Lambda_{x_1}(N),\Lambda_{x_2}(N),\dots)$ of
renormalized ``good'' sites, with
$x_k=(x_{k,1},\dots,x_{k,d})\in\Z^d,\|x_{k+1}-x_k\|_1=1,\forall
k\in\N$, and $x_0=(0,\dots,0)$. From now on, we fix some
configuration $\o\in\Omega$ for which this infinite path
$(\Lambda_{x_0}(N),\Lambda_{x_1}(N),\Lambda_{x_2}(N),\dots)$ of
renormalized ``good'' sites occurs.

Given any word $\xi=(\xi_1,\xi_2,\cdots)\in\Xi$, we can see its
digits along some path $\gamma =\langle v_0=0,v_1,v_2\dots\rangle$
starting from the origin of the original lattice in the following
way.

Define $v_0$ as being the origin, we will define the others
vertices inductively. Given the vertex
$v_{k-1}\in\L_{x_{k-1}}(N)$, let $i_k\in\{1,\dots,d\}$ be the
unique integer such that $|x_{k-1,i_k}-x_{k,i_k}|=1$. Since the box
$\L_{x_k}(N)$ is good there exists at least one vertex
$v\in\L_{x_k}(N)$ along the line $L(i_k,(l_j)_j)$ with
$l_j=v_{k-1,j},\forall j\neq i_k$ such that $X(v)=\xi_k$. Choose
one of these vertices and call it by $v_k$. Observe that $v_{k-1}$
and $v_k$ belong to the same line and $\|v_{k-1}-v_k\|_1\leq 2N-1,
\forall k\in\N$.

Then by construction, on this fixed configuration $\o$, we have $\xi_k=X(\o,v_k),$
$\forall k\in\N$. So, taking
$K(p)=2N(p)-1$ we have that
$$\P_p\{\omega\in\Omega; S_0(\omega)=\Xi\mbox{ on }G_K\}>0.$$

The last statement of the theorem follows by observing that
the event $\cup_{v\in\V}\{\omega\in\Omega;
S_v(\omega)=\Xi\mbox{ on }G_K\}$ is translation invariant, so its
probability must be 0 or 1.

\end{proof}

A natural question one could ask is about the magnitude of $K(p)$. The truncation constant $K(p)$ has its minimum at $p=\frac{1}{2}$
(when $d=2$, the constant $K(\frac{1}{2})$ could be taken as 11)
and increases to infinity as $p$ approaches $0$ or $1$. One
problem of relevance is to determine how $K(p)$ scales as $p$ goes
to zero. Without loss of generality (by symmetry) we consider
only the situation where $p\in (0,\frac{1}{2}]$.

Related problems on other models have been extensively studied,
for example, in \cite{AL}, the authors determine the right finite
size scaling as $p$ goes to zero for the critical threshold $2D$
Bootstrap Percolation. This is the setup of the next theorem and lemmas.

\begin{lema}\label{super}
 If $K=K(p)=2\lfloor\frac{\l}{p}\rfloor$, then for $\l>-3\ln (1-p_c(\Z^d))$ it holds that
\be \lim_{p\ra
0}\P_p\left(\cup_{v\in\V}\{\omega\in\Omega; S_v(\omega)=\Xi\mbox{ on }G_K\}\right)=1.\ee
\end{lema}

\begin{proof}By translation invariance, it is enough to prove that there exists some $p^*>0$, such that for large $\l$ and for all $ p\in(0,p^*)$, 
\be\label{final}
\P_p\{\omega\in\Omega; S_0(\omega)=\Xi\mbox{ on }G_K\}>0.
\ee

We say that there is a {\it seed} at vertex $v\in\Z^d$, if $X(v)=1$ and $X(u)=0, \forall u$ with $\|v-u\|_1=1$. We call the vertex $v$ the {\it center} of the seed. Observe that $\P_p\{\mbox{there is a seed located at }v\}=(1-p)^{2d}p$ and the events $\{\mbox{there is a seed located at }v_1\}$ and $\{\mbox{there is a seed located at }v_2\}$ are independent if $\|v_1-v_2\|_1\geq 3$.

As $\lim_{p\ra 0}[1-[1-p(1-p)^{2d}]^{\lfloor\frac{1}{3}\lfloor\frac{\l}{p}\rfloor\rfloor}]=1-\exp(-\frac{\l}{3})$, we can choose some large $\l>-3\ln (1-p_c(\Z^d))$ and some small $p^*$, such that
\be\label{cotainferior}
1-[1-p(1-p)^{2d}]^{\lfloor\frac{1}{3}\lfloor\frac{\l}{p}\rfloor\rfloor}>\gamma>p_c(\Z^d),\ \forall p\in(0,p^*).
\ee
Fixed this large $\l$ and $ p\in(0,p^*)$, define $n=\lfloor\frac{\l}{p}\rfloor$ and consider the partition of $\Z^d$, $\{\L_x(n);
x\in\Z^d\}$ as defined in (\ref{particao}). We will use the letters $x$ and $y$ to denote vertices of the renormalized lattice. The idea is to construct, dynamically, a sequence $(R_x, x\in\Z^d)$, of $\{0,1\}$ valued random variables and a sequence $(D_i,E_i), i=0,1,\dots$, of ordered pairs of subsets of $\Z^d$, defined as follows:

First of all, let $f:\N\ra\Z^d$ be a fixed ordering of the vertices of $\Z^d$ and define $(D_0,E_0)=(\emptyset,\emptyset)$. Let $x_0=0$ be the origin of $\Z^d$. 	We say that $R_{x_0}=1$ if at least one of the $d(n-1)$ vertices in the set 
\[T_{x_0}=\{v=(v_1,\dots,v_d)\in\Z^d; \exists i\in\{1,\dots,d\}\mbox{ with }v_i\in\{1,\dots,n-1\}\]
\[\mbox{ and } v_j=0,\forall j\neq i\}\] is the center of some seed, that is if $\exists v\in T_{x_0}$ with $X(v)=1$ and $X(u)=0, \forall u$ with $\|v-u\|_1=1$. Otherwise, we say that $R_{x_0}=0$. Observe that $$\P_p(R_{x_0}=1)\geq 1-[1-p(1-p)^{2d}]^{\lfloor\frac{d(n-1)}{3}\rfloor}\geq 1-[1-p(1-p)^{2d}]^{\lfloor\frac{n}{3}\rfloor}.$$ Now, define

\begin{equation}
(D_1,E_1)=\left\{
\begin{array}
[c]{l}
(D_0\cup\{x_0\},E_0) \mbox{  if } R_{x_0}=1,\\
(D_0,E_0\cup\{x_0\}) \mbox{  if } R_{x_0}=0,
\end{array}\right.
\end{equation}
and if $R_{x_0}=1$ define $z(x_0)$ as the center of some seed belonging to $T_{x_0}$.

Let $(D_i,E_i)$ be given. If $\partial_e(D_i)\cap E_i^c=\emptyset$, define $(D_j,E_j)=(D_i,E_i),\ \forall j>i$, where $$\partial_e(A)=\{v\in\Z^d; v\in A^c\mbox{ and }\exists u\in A\mbox{ with }\|v-u\|_1=1\}.$$ Otherwise, let $x_{i}$ be the first vertex in the fixed order belonging to $\partial_e(D_i)\cap E_i^c$ and define $y_{i}$ as any vertex belonging to $D_i$ such that $\|x_{i}-y_{i}\|_1=1$ (observe that $y_1=x_0$).

We say that $R_{x_{i}}=1$, if at least one of the $n$ vertices of the set
$$T_{x_{i}}=\L_{x_{i}}\cap\{z(y_{i})+ j\bar{e}_{x_{i}-y_{i}};j\in\Z\}$$ is the center of some seed, that is if $\exists v\in T_{x_{i}}$ with $X(v)=1$ and $X(u)=0, \forall u$ with $\|v-u\|_1=1$. Here, $\bar{e}_l$ denotes the unit vector of $\Z^d$ in the $l$-th direction. Otherwise, we say that $R_{x_{i}}=0$. Observe that $\P_p(R_{x_{i}}=1|R_{x_j},\ \forall j< i)\geq 1-[1-p(1-p)^{2d}]^{\lfloor\frac{n}{3}\rfloor}$.
Define

\begin{equation}
(D_{i+1},E_{i+1})=\left\{
\begin{array}
[c]{l}
(D_i\cup\{x_i\},E_i)\mbox{  if } R_{x_i}=1,\\
(D_i,E_i\cup\{x_i\})\mbox{  if } R_{x_i}=0,
\end{array}\right.
\end{equation}
and if $R_{x_i}=1$ define $z(x_i)$ as some center of seed belonging to $T_{x_i}$.
Due our choice of $\l, p$ and $n$, the process $(R_x, x\in\Z^d)$ dominates an i.i.d. $\{0-1\}$ valued process with parameter large than $p_c(\Z^d)$. Comparison with ordinary site percolation shows that (see Lemma 1 in \cite{GM}) $\P_p(\#(\cup_{i\in\N}D_i)=\infty)>0$ and by construction, on the event $(\#(\cup_{i\in\N}D_i)=\infty)$, all words $\xi\in\Xi$ can be seen along some self-avoiding path $\langle 0,v_1,v_2,\dots\rangle$ with $v_i$ belonging to some seed $\forall i$, as we will now show. Then, (\ref{final}) is proved with $K(p)=2\lfloor\frac{\l}{p}\rfloor$.

When the event $\{\#(\cup_{i\in\N}D_i)=\infty\}$ occurs, it is possible to take a sequence of adjacent boxes $\L_{x_{i_0}}, \L_{x_{i_1}}, \L_{x_{i_2}},\dots$, with $x_{i_0}=x_0=0$, such that $R(x_{i_j})=1,\forall j$ and 
$z(x_{i_j})-z(x_{i_{j-1}})=m\bar{e}_l$ for some $m\in\Z^*$ and $l\in\{1,\dots,d\}$. That is, seeds in adjacent boxes have their centers belonging to the same line. To simplify the notation, let us denote $x_{i_j}$ by $w_j$.

Given any word $\xi\in\Xi$, define $l_1=\min\{i;\xi_i=1\}$ and $l_j=\min\{i>l_{j-1};\xi_i=1\}$ for $j\geq 2$. If $l_1=1$, define $v_1=z(w_0)$; if $l_1>1$, define $v_i=z(w_{i-1})-\bar{e}_b,\forall i<l_1$ and $v_{l_1}=z(w_{l_1-2})$, where $b$ is the unique direction such that the inner product $\langle\bar{e}_b\cdot z(w_0)\rangle$ is not zero. Then, by construction, the finite word $(\xi_1,\dots,\xi_{l_1})$ is seen along the path $\langle 0,v_1,\dots,v_{l_1}\rangle$. Define $I(1)$ as the index such that $v_{l_1}=z(w_{I(1)})$ (observe that $I(1)=0$ if $l_1=1$ and $I(1)=l_1-2$ if $l_1\geq 2$).

Now, we describe the induction step. Suppose that the finite word \linebreak
$(\xi_1,\dots,\xi_{l_k})$ is seen along the path $\langle 0,v_1,\dots,v_{l_k}\rangle,\ \forall k\geq 1$.
If $l_{k+1}=l_k+1$, define $v_{l_{k+1}}=z(w_{I_k+1})$; if $l_{k+1}>l_k+1$, define $v_i=z(w_{I(k)+i-l_k})-\bar{e}_{w_{I(k)}-w_{I(k)+1}},\forall l_k<i<l_{k+1}$ and $v_{l_{k+1}}=z(w_{I(k)+l_{k+1}-1+l_k})$. Then, by construction, the finite word $(\xi_1,\dots,\xi_{l_{k+1}})$ is seen along the path $\langle 0,v_1,\dots,v_{l_{k+1}}\rangle$. Define $I(k+1)$ as the index such that $v_{l_{k+1}}=z(w_{I(k+1)})$ (observe that $I(k+1)=I(k)+1$ if $l_{k+1}=l_k+1$ and $I(k+1)=I(k)+l_{k+1}-l_k-1$ if $l_{k+1}>l_k+1$).

Thus, we define the path $\langle 0,v_1,v_2,\dots\rangle$, in such way that $X(v_i)=\xi_i, \forall i$. This finishes the proof of the lemma.

\end{proof}

\begin{lema}\label{sub} If $K=K(p)=\lfloor\frac{\l}{p}\rfloor$ with
$\l<\frac{1}{2d}$, it holds that \be \lim_{p\ra
0}\P_p\left(\cup_{v\in\V}\{\omega\in\Omega; S_v(\omega)=\Xi\mbox{
on }G_K\}\right)=0.\ee
\end{lema}
\begin{proof}For the subcritical behavior, with a standard
argument we show that for $\l<(2d)^{-1}$ the word $\bar
1=(1,1,...)$ does not percolate. Let $\sigma_m^K$ be the number of
self-avoiding paths of length $m$ starting from the origin on the
graph $G_K$, and let $M_m^K$ be the number of such paths which are
occupied. It is clear that, if we see the word $\bar{1}$ from the
origin, then there are occupied paths of all lengths starting from
the origin. This implies that, $\forall\, m\in\mathbb{N},$

\[ \P_p\{\o\in\Omega;\bar{1}\in S_0(\o)\mbox{ on
}G_{K}\}\leq\P_p\{\o\in\Omega;M_m^K(\o)\geq1\mbox{ on }G_{K}\}\leq
\]
\[ \Ed(M_m^K)=p^m\sigma_m^K\leq(p2dK)^{m}. \]

This last inequality follows from the fact that, in order
to have a self-avoiding path, each new step has at most $2dK$
choices. Therefore,
$$\P_p\{\o\in\Omega;\bar{1}\in S_0(\o)\mbox{ on
}G_{K}\}\leq \lim_{m\rightarrow\infty} (p2dK)^{m}.$$ Thus, if
$K<\frac{1}{p2d}$, it holds that $$\P_p\{\o\in\Omega;\bar{1}\in
S_0(\o)\mbox{ on }G_{K}\}=0,$$ that is,
$$\P_p\left(\cup_{v\in\V}\{\omega\in\Omega; S_v(\omega)=\Xi\mbox{
on }G_K\}\right)=0.$$
\end{proof}

\begin{teorema}\label{fss}There exists a constant $\l_0\in\left(\frac{1}{2d},-6\ln (1-p_c(\Z^d))\right)$ such that if $K(p)=\lfloor\frac{\l}{p}\rfloor$, it holds that
\begin{equation}
\lim_{p\ra 0}\P_p\left(\{\omega\in\Omega; S_v(\omega)=\Xi\mbox{  on }G_K\}\right)=\left\{
\begin{array}
[c]{l}
0 \mbox{  if } \l<\l_0,\,\\
1 \mbox{  if } \l>\l_0.
\end{array}
\right.
\end{equation}
\end{teorema}
\begin{proof}Observe that $\P_p\left(\{\omega\in\Omega; S_v(\omega)=\Xi\mbox{  on }G_K\}\right)$ is increasing in $\l$ and must be 0 or 1 by translation invariance. Therefore, this theorem follows by Lemmas \ref{super} and \ref{sub}. 
\end{proof}

Observe that in Lemma \ref{super} we made a more involved construction than in Theorem \ref{box}. The reason is that the right scale for $K(p)$ is different if we consider the event percolation of good boxes, like is shown in the next theorem.

\begin{teorema}  Let $A_0(n)=\{\o\in\Omega;\mbox{the box}\;\L_0(n)\mbox{ is good in }\o\}.$ Then, for $n=n(p)=\lfloor\frac{-\beta\ln p}{p}\rfloor$, we have that

\be\label{escala} \displaystyle\lim_{p\ra 0}\P_p\left(
A_0(n)\right)=
\begin{cases}1 \mbox{ if }\beta >d-1,\\ 0 \mbox{ if } \beta \leq
d-1.
\end{cases}
\ee
\end{teorema}

\begin{proof}
For $i\in\{1,\dots,d\}$ define the events
\[C_0^i(n)=\{\o\in\Omega :\forall (l_j)_j \mbox{ with }
l_j\in\{1,\dots ,n\} \mbox{ and } j\in\{1,\dots,d\}-\{i\}\]
\[\mbox{there exists }z\in L(i,(l_j)_j)\mbox{  such that }X(\o,z)=1\}\]
and $$B_0(n)=\cap_{i=1}^d C_0^i(n),$$ where
$$L(i,(l_j)_j)=\{y\in\L_0(n);y_j=l_j,\forall
j\in\{1,\dots,d\}-\{i\}\}$$ are the lines of $\L_0(n)$.

By definition of $A_0(n)$ and  $B_0(n)$, we have that
\be\label{tira0}\displaystyle\lim_{p\ra 0}\P_p\left(
B_0(n)\backslash A_0(n)\right)=0.\ee For all $i\in\{1,\dots,d\}$, $C_0^i(n)$
are increasing events, so by the FKG inequality and rotational
invariance we have that \be\label{FKG}[\P_p(C_0^1(n))]^d\leq
P_p(B_0(n))\leq P_p(C_0^1(n)).\ee

Thus, using (\ref{tira0}) and (\ref{FKG}) it is enough to prove
(\ref{escala}) replacing the event $A_0(n)$ by $C_0^1(n)$.

Observe that
$$\P_p\left(C_0^1(n)\right)= [1-(1-p)^n]^{n^{d-1}}.$$

Then, when $n=n(p)=\lfloor\frac{-\beta\ln p}{p}\rfloor$ we have
$$\lim_{p\ra 0}\P_p\left( C_0^1(n)\right)=\lim_{p\ra 0}[1-(1-p)^{\frac{-\beta\ln p}{p}}]
^{({\frac{-\beta\ln p}{p}})^{d-1}}$$ 
$$=\lim_{p\ra 0}\exp [-(-\beta\ln
p)^{d-1}p^{\beta-(d-1)}]=
\begin{cases}1, \mbox{ if }\beta >d-1,\\ 0, \mbox{ if } \beta \leq
d-1
\end{cases}$$
\end{proof}







{\bf Remarks:}
\begin{description}
\item {i)} All results of this section remain valid if we
consider any finite alphabet instead
of the binary alphabet, that is  $\Xi=\{0,1,\dots, n-1\}^\N$.
\item{ii)} The statement of Theorem \ref{fss} remains the same replacing the event $\cup_{v\in\V}\{\omega\in\Omega;S_v(\omega)=\Xi\mbox{ on }G_K\}$ by $\{\omega\in\Omega;S_\infty(\omega)=\Xi\mbox{ on }G_K\}$.

\end{description}
Nevertheless, in both cases, the constant $\l_0$ should be different.

\section{Percolation of random words}\label{quasetodas}

Now we consider the same kind of scaling question, but concerning
the probability 

$$\P_p\left(\cup_{v\in\V}\{\omega \in \Omega;\mu_\a(S_v(\omega))=1)\mbox{ on }G_K\}\right),$$
i.e., the probability that almost all words are seen on $G_K$ from one vertex.

We aim at proving an analogue of Theorem
\ref{fss}. Observe that when
$\alpha =0$ we have ordinary percolation of 0's, and so the constant
$K$ can be taken equal to 1. When $\alpha=1$ the right scale of
$K(p)$ is the same as in Theorem \ref{fss} (see Corollary \ref{111} in the final remarks).  We are not yat able to determine the right scale, actually,
we don't even know if the scale itself changes (as the next
theorem might suggest) or if only the constant $\l_0$ would change, but
we can give an lower bound.

\begin{teorema}\label{fss2}
Given $0<\a<1$, we have that for all $\epsilon>0$ and
$K(p)<\frac{1}{p^{\a-\epsilon}}$, it holds that
$$\displaystyle\lim_{p\ra 0}
\P_p\left(\cup_{v\in\V}\{\omega \in \Omega;\mu_\a(S_v(\omega))=1)\mbox{ on }G_K\}\right)=0.$$
\end{teorema}

\begin{proof}
Given $\epsilon>0$ and $N_0\in\N$, consider the following subset
of words
$$A_{N_0}^\e=\{\xi\in\Xi ; \Big| \frac{\sum_{i=1}^n\xi_i}{n}-\a
\Big|<\e, \forall n\geq N_0\}.$$ We claim that
$\mu_\a(A_{N_0}^\e)\rightarrow 1$ as $N_0\rightarrow\infty$ . To
see this, note that, for all $N_0$, $A_{N_0}^\epsilon \subset
A_{N_{0}+1}^\epsilon$. This implies that $A_{N_0}^\epsilon
\uparrow
A_\infty^\epsilon=\displaystyle\left(\bigcup_{N_0=1}^\infty
A_{N_0}^\epsilon\right)$ and $\mu_\a(A_{N_0}^\e)\rightarrow \mu_\a
(A_\infty^\epsilon)$ as $N_0\rightarrow\infty$. By the Strong Law
of Large Numbers, for all $\xi$, there exists, $\mu_\alpha -a.s.$,
some $n_0(\xi)\in\N$ such that,
$$\Big|\frac{\sum_{i=1}^n\xi_i}{n}-\a \Big|<\e,\;\;\;\; \forall\; n\geq
n_0.$$
This implies that $\mu_\a(A_\infty^\epsilon)=1$.

On the set $A^\e_{N_0}$, we have \be \label{lln}(\a-\e)n\leq
\sum_{i=1}^n\xi_i\leq (\a+\e)n\ee for all $n\geq N_0$.

Given any $\xi\in\Xi$, we will denote
$\xi^{(n)}=(\xi_1,\dots,\xi_n)$. Then, for any $n\geq N_0$, we
have

\be\{\omega\in\Omega;S_0(\o)\cap A^\e_{N_0}\not=\emptyset\mbox{ on $G_K$ }\}\ \ee 
\[\subset\bigcup_{\cima{\g;|\g|=n}{\xi^{(n)}; \xi\in A^\e_{N_0}}}
\{\omega\in\Omega;\xi^{(n)}\mbox{ is seen in }\o\mbox{ along the
path }\g\mbox{ on }G_K\},
\]
where the union is over all self avoiding paths on $G_K$ of size
$n$ having the origin as it starting point. Hence for all $n\geq N_0$,

\be\P_p\{\omega\in\Omega;S_0(\o)\cap A^\e_{N_0}\not=\emptyset \mbox{ on $G_K$ }\}\leq \sum_{\cima{\g;|\g|=n} {\xi^{(n)}; \xi\in A^\e_{N_0}}}
p^{\sum_{i=1}^n\xi_i}(1-p)^{n-\sum_{i=1}^n\xi_i}. \ee

Using (\ref{lln}), we have that for all $n\geq N_0$
\be
\P_p\{\omega\in\Omega;S_0(\o)\cap A^\e_{N_0}\not=\emptyset\mbox{ on $G_K$ }\} \leq
(2dK)^n2^np^{(\a-\e)n}(1-p)^{n-(\a+\e)n}. \ee

Thus, taking $K<(2dp^{\a-\e})^{-1}$ and observing that $(1-p)^{\a+\e-1}<2$ for sufficiently small $p$, we have  $4dKp^{\a-\e}(1-p)^{1-\a-\e}<1$, and so

\be\label{wordsnotseen}\P_p\{\omega \in
\Omega;\,(A^{\epsilon}_{N_0}\cap S_0(\o))=\emptyset\mbox{ on
}G_K\}=1,\,\,\,\forall\,N_0\in \N.\ee

Now, we claim that
$$\P_p\{\omega \in \Omega;\,(A_{\infty}^\epsilon\cap
S_0(\o))=\emptyset\mbox{ on }G_K\}=1.$$ To see this,
let
$$Z_{N_0}=\{\omega \in \Omega;\,(A_{N_0}^\epsilon \cap
S_0(\omega))=\emptyset\mbox{ on }G_K\},$$ and observe that
$\{Z_{N_0}\}_{N_0\geq1}$ is a decreasing sequence. This implies
that $Z_{N_0}\downarrow \displaystyle\bigcap_{N_0=1}^{\infty}
Z_{N_0}=Z_\infty$ and
$$\P_p(Z_\infty)=\displaystyle\lim_{N_0 \rightarrow \infty}\P_p(Z_{N_0})=1.$$
It remains to show that $Z_\infty=\{\omega \in
\Omega;\,(A_\infty^\epsilon \cap S_0(\omega))=\emptyset\mbox{ on
}G_K\}$. But this follows from the implications

\[\omega \in \displaystyle\bigcap_{N_0=1}^{\infty} Z_{N_0}\Leftrightarrow \forall\,
N_0 \in \N,\,A^{\epsilon}_{N_0}\cap S_0(\omega)= \emptyset\]
\[\Leftrightarrow\displaystyle\bigcup_{N_0=1}^\infty(A^{\epsilon}_{N_0}\cap
S_0(\omega))=\emptyset \Leftrightarrow A^{\epsilon}_\infty \cap
S_0(\omega)=\emptyset.\]

This implies that
$$\P_p\{\omega \in \Omega;\,(A_\infty^\epsilon\cap S_0(\o))=\emptyset\mbox{
on }G_K\}=1.$$ As $\mu_\a(A_\infty^\epsilon)=1$, we conclude that
$$\P_p\{\omega \in \Omega;\,\mu_\a(S_0(\o))=1\mbox{
on }G_K\}=0,$$ or equivalently
$$\P_p\{\omega \in \Omega;\exists v\in\V\mbox{ with }
\mu_\a(S_v(\omega))=1)\mbox{ on }G_K\}=0.$$

\end{proof}

\section{Final Remarks}\label{sfinal}

As a straightfoward corollary of Lemmas \ref{super} and \ref{sub}, we obtain the precise scaling behavior of the truncation constant $K(p)$ as $p$ goes to zero for ordinary percolation.

\begin{corolario}\label{111} There exists a constant $\l_0\in\left(\frac{1}{2d},-2\ln(1-p_c(\Z^d))\right)$ such that if $K(p)=\lfloor\frac{\l}{p}\rfloor$, it holds that
\begin{equation}
\lim_{p\ra 0}\P_p\{\omega\in\Omega; (1,1,\dots,1,\dots )\mbox{ is seen in }\omega\mbox{  on }G_K\}=\left\{
\begin{array}
[c]{l}
0 \mbox{  if } \l<\l_0\,\\
1 \mbox{  if } \l>\l_0.
\end{array} \right.
\end{equation}
\end{corolario}
\begin{proof}It is enough to observe that $$\P_p\{\omega\in\Omega; (1,1,\dots,1,\dots )\mbox{ is seen in }\omega\mbox{  on }G_{\lfloor\frac{\l}{p}\rfloor}\}$$ is increasing in $\l$ and must be 0 or 1, by translation invariance. Lemma \ref{sub} says that $\l_0>\frac{1}{2d}$ and with a simple modification of the proof of Lemma \ref{super} we can show that $$\l_0<-\ln(1-p_c(\Z^d)).$$ In \cite{K}, it is shown that $\lim_{d\ra\infty} 2dp_c(\Z^d)=1$. Therefore this constant $\l_0$ must be such that $\frac{1}{2}<d\l_0<1$ when $d$ is large. 
\end{proof}

Related to the comment above, a natural question one could ask is:

{\bf  Question}\ \  Determine the asymptotic behaviour, in the dimension $d$, of $\l_0$ in Theorem \ref{fss}.

As we mentionned before, it is an open question if for the events $\{\mu_\alpha(S_\infty)=1\})$ and $\cup_{v\in\V}\{\mu_\alpha(S_v)=1\}$ there are  precise results like Theorem \ref{fss}. As a matter of fact, we believe that:

{\bf Conjecture}\ \  For any $\e>0$, let $K(p)=\lfloor\frac{1}{p^{\a+\e}}\rfloor$, then 
$$\lim_{p\ra 0} \P_p\left(\cup_{v\in\V}\{\omega \in \Omega;\mu_\a(S_v(\omega))=1)\mbox{ on }G_K\}\right)=1.$$

This would not, however, answer completely the precise scaling behaviour of the event above. Indeed, one could expect the following:

{\bf Conjecture}\ \  There is a $\l_0\in (0,\infty)$ such that, if $K(p)=\lfloor\frac{\l}{p^\a}\rfloor$, the following limit holds

$$\lim_{p\ra 0} \P_p\left(\cup_{v\in\V}\{\omega \in \Omega;\mu_\a(S_v(\omega))=1)\mbox{ on }G_K\}\right)=\left\{
\begin{array}
[c]{l}
0,\mbox{  if } \l<\l_0\,\\
1,\mbox{  if } \l>\l_0
\end{array}\right.
.$$

Of course, the first conjecture is implied by the latter. Related to the conjectures above, one could ask the following:

{\bf Question}\ \  Is the threshold scaling for the event 
$\{\omega\in\Omega ; \mu_\alpha
(S_\infty(\omega))=1\mbox{ on }G_K\}$, the same as for the event 
$\{\omega \in \Omega;\exists v\in\V\mbox{ with }
\mu_\a(S_v(\omega))=1)\mbox{ on }G_K\}$ or is it strictly smaller?

 

{\bf Acknowledgments.} During the preparation of this manuscript authors exchange correspondence with G. Grimmett who also discovered proof of Theorem \ref{box}, we thank him for comments. B.N.B.de Lima is partially supported by
CNPq and FAPEMIG (Programa Pesquisador Mineiro), R. Sanchis is
partially supported by CNPq and R.W.C.Silva is partially supported
by FAPEMIG.

\end{document}